\documentclass[10pt, reqno]{amsart}
\usepackage[active]{srcltx} 
\usepackage{enumerate}
\usepackage{amsmath}
\usepackage{amssymb}
\usepackage{amsfonts}
\usepackage{amsthm}
\usepackage[usenames]{color}
\usepackage{version}
\usepackage{hyperref}
\allowdisplaybreaks

 \newtheorem{theorem}{Theorem}[section]
 
 \newtheorem{lemma}[theorem]{Lemma}

\newtheorem{definition}[theorem]{Definition}
\newtheorem{remark}[theorem]{Remark}
\newtheorem{fact*}{Fact}

\DeclareMathOperator{\RE}{Re}

\newcommand\dd{\mathrm d}

\newcommand{\M}{\mathcal{M}}
\newcommand{\N}{\mathcal{N}}
\newcommand{\T}{\mathbb{T}}

\newcommand{\D}{\mathbb{D}}
\newcommand{\C}{\mathbb{C}}

\newcommand{\schur}{\mathcal{S}_2}

\newcommand{\Lop}{\mathcal{L}}
\newcommand{\cc}[1]{\overline{#1}}
\newcommand{\abs}[1]{\left\vert#1\right\vert}

\newcommand{\norm}[1]{\left\Vert#1\right\Vert}

\newcommand{\nt}{\stackrel{\mathrm {nt}}{\to}}

\newcommand{\ip}[2]{\left\langle #1, #2 \right\rangle}
\newcommand{\ad}{^\ast}
\newcommand{\inv}{^{-1}}

\newcommand{\til}{\raise.17ex\hbox{$\scriptstyle\mathtt{\sim}$}}

\newcommand{\vp}{\varphi}
\newcommand{\ph}{\varphi}

\newcommand\la{\lambda}

\newcommand\beq{\begin{equation}}

\newcommand\eeq{\end{equation}}

\newcommand{\vectwo}[2]
{
   \begin{pmatrix} #1 \\ #2 \end{pmatrix}
}

\newcommand\bbm{\begin{bmatrix}}
\newcommand\ebm{\end{bmatrix}}
\newcommand\bpm{\begin{pmatrix}}
\newcommand\epm{\end{pmatrix}}
\numberwithin{equation}{section}

\newlength{\Mheight}
\newlength{\cwidth}
\newcommand{\mc}{\settoheight{\Mheight}{M}\settowidth{\cwidth}{c}M\parbox[b][\Mheight][t]{\cwidth}{c}}

\title[Analytic functions at boundary singularities]{Analytic functions on the bidisk at boundary singularities via Hilbert space methods}
\author{R. Tully-Doyle}
\address{Ryan Tully-Doyle, Department of Mathematics, Hampton University, Hampton VA, 23668}
\email{ryan.tullydoyle@hamptonu.edu}
\date{\today}

\begin{document}

\begin{abstract}
We investigate the behavior of a generalized Hilbert space model of a function in the Schur class of the bidisk at singular boundary points that satisfy a growth condition. We examine the relationship between the boundary behavior of Schur functions and the geometry of corresponding generalized Hilbert space models. We describe a geometric condition on an associated operator that classifies the behavior of the directional derivative of the underlying Schur function at a carapoint.
\end{abstract}

\maketitle

The \emph{Schur class} in one variable, denoted by $\mathcal S$, is the set of analytic functions $\ph \in \mathcal S$ that map the complex unit disk $\D$ into itself. Beginning in the early 20th century, analysts studied the Schur class and conformally related families of functions. A classical theorem due to C. Carath\'edory and R. Julia from this period relates the differentiability of Schur functions at boundary points to a regularity condition at the boundary \cite{ju20, car29}. 

In this paper, we consider the two variable \emph{Schur-Agler class}, denoted by $\schur$. A function $\ph$ is in $\schur$ if $\ph$ is an analytic map of the bidisk $\D^2$ into $\D$. In two variables, the situation is complicated by the existence of nontrivial singular sets at the distinguished boundary of the bidisk, the torus $\T^2$, even for rational functions. Despite this obstruction, it is possible to formulate a version of the classical theorem in several variables (see, e.g. \cite{ab98, bgr90, rud80, wlo87}).  In particular, in a paper of 2010, Agler, {\mc}Carthy, and Young generalized the classical theorem to two variables by way of an operator theoretic construct called a Hilbert space model. Beyond giving a natural generalization of the one variable case, Agler, {\mc}Carthy and Young's theorem characterized the boundary behavior of two variable Schur functions in terms of the objects in the Hilbert space model. 

In \cite{aty12}, the author, with J. Agler and N.J. Young, developed a generalized Hilbert space model particularly suited to the study of the behavior of rational functions at boundary singularities, at the cost of losing the ability to use operator theoretic conditions at certain boundary singularities to detect the differential structure in the function being modeled. 

We first develop the notion of \emph{singular} and \emph{regular} generalized models by looking at a geometric condition on the model Hilbert space. With these definitions, our main results in this paper, Theorem \ref{forward2} and Theorem \ref{diffmain}, characterize the differential structure of a Schur function at a singular boundary point in terms of generalized models, recovering the spirit of the two variable Julia-Carath\'eodory Theorem in \cite{amy10a}.

The central object in generalized Hilbert space models is an operator-valued rational inner function in two variables. In \cite{knese14}, G. Knese describes boundary behavior of rational inner functions from the bidisk into the disk. In \cite{pascoe16}, J. E. Pascoe develops a method for constructing rational inner functions of a given level of regularity at the boundary. We anticipate that this work will lead to further extension of the generalized Hilbert space model approach to a larger set of boundary singularities.

The author would like to thank N.J. Young for support and for providing a key insight \cite{young12}. 

\section{Preliminaries}
\subsection{Carapoints}

 For a function $\ph \in \schur$, points that satisfy the following \emph{Carath\'eodory condition} are called carapoints \cite{aty12}. 

\begin{definition}\label{cara1}
Let $\ph \in \schur$. A point $\tau \in \T^2$ is a carapoint for $\ph$ if there exists a sequence $\{\la_n\} \subset \D^2$ tending to $\tau$ such that 
\beq
\frac{1 - \abs{\vp(\la)}}{1 - \norm{\la}_\infty} \text{ is bounded.}
\eeq
\end{definition}

In the bidisk, a set $S$ approaches $\tau$ nontangentially if there exists a positive constant $c$ so that for all $\la \in S$,
\[
\norm{\tau - \la}_\infty \leq c(1 - \norm{\la}_\infty),
\]
where $\norm{\la}_\infty = \max\{\abs{\la^1}, \abs{\la^2}\}$. A sequence $\{\la_n\}$ is said to approach $\tau$ nontangentially, that is $\la_n \nt \tau$, if $\{\la_n\} \subset S$ for some set $S \subset \D^2$ that approaches $\tau$ non-tangentially.

\subsection{Models}

A primary tool used to study the boundary behavior of functions in $\schur$ is a \emph{Hilbert space model}.
\begin{definition}\label{model}
Let $\vp \in \schur$. A pair $(\M, u)$ is a \emph{model} for $\vp$ if $\M = \M_1 \oplus \M_2$ is an orthogonally decomposed separable Hilbert space and $u:\D^2 \to \M$ is an analytic map such that 
\beq \label{modeleq}
1 - \cc{\vp(\mu)}\vp(\la) = \ip{(1 - \cc\mu^1\la^1)u_\la}{u_\mu}_{\M_1} + \ip{(1 - \cc\mu^2\la^2)u_\la}{u_\mu}_{\M^2}
\eeq
holds for every $\la, \mu \in \D^2$, where $u_\la = u(\la)$. Abusing notation slightly within the inner product on $\M$, if we let $\la$ inside the inner product represent the operator on $\M$ given by
	$$\la = \la^1 P_{\M_1} + \la^2 P_{\M_2},$$
then \eqref{modeleq} can be written in compressed notation as
\beq \label{model2}
1 - \cc{\vp(\mu)}\vp(\la) = \ip{(1 - \mu\ad\la)u_\la}{u_\mu}_\M.
\eeq
\end{definition}
Every function in $\schur$ has a model \cite{ag90, ampi}. 

Hilbert space models and realizations encode function theoretic data about Schur-Alger functions into the structure of a Hilbert space and associated maps. 

\begin{definition} \label{modelpoints}
For a given function $\vp \in \schur$, a point $\tau \in \D^d$ is a $B$-point of the model if $u$ is bounded on every subset of $\D^d$ that approaches $\tau$ nontangentially. The point $\tau$ is a $C$-point of the model if, for every subset $S$ of $D^d$ that approaches $\tau$ nontangentially, $u$ extends continuously to $S \cup \{\tau\}$ (with respect to the norm topology on $\M$).
\end{definition}

In \cite{amy10a}, Agler, {\mc}Carthy, and Young used Hilbert space model techniques to generalize the classical Carath\'eodory-Julia Theorem to two variables in terms of the properties of a model at a boundary point. The following theorems represent a qualitative version of those results.

\begin{theorem}[Agler, {\mc}Carthy, Young] \label{bpoint}
Let $\ph \in \schur$, and $\tau \in \T^2$. The following are equivalent:
 \begin{enumerate}
 \item $\tau$ is a carapoint for $\ph$;
 \item there exists a model $(\M, u)$ of $\ph$ such that $\tau$ is a $B$-point;
 \item for every model $(\M, u)$ of $\ph$, $\tau$ is a $B$-point.
\end{enumerate}
\end{theorem}

\begin{theorem}[Agler, {\mc}Carthy, Young]\label{modelbpoint}
 If $\tau$ is a $B$-point for a model $(\M, u)$ of $\ph$, then the nontangential limit of $\ph$ at $\tau$ given by
	$$\ph(\tau) := \lim_{\la \nt \tau} \ph(\la)$$
exists.
\end{theorem}

\begin{theorem}[Agler, {\mc}Carthy, Young]\label{modelcpoint}
 $\tau$ is a $C$-point for a model $(\M, u)$ of $\ph$ if and only if $\ph$ is nontangentially differentiable at $\tau$.
\end{theorem}

That is, boundedness and continuity of the model function $u_\la$ at a boundary point characterizes the boundary behavior of the Schur function at that point. More can be said about the differential structure of functions at carapoints (the subject of \cite{amy10a}), which will be discussed in the following sections.

\section{Generalized models and directional derivatives}

In \cite{aty12}, the author with J. Agler and N. Young developed a generalized model for functions in $\schur$ with a singular carapoint $\tau \in \T^2$, where the operator $\la$ in \eqref{modeleq} is replaced by a contractive operator-valued map $I$ defined in terms of a positive contraction on a Hilbert space. In the case that $\ph$ has a singular carapoint at $\tau$, $I_Y$ models the behavior of the singularity. We first introduce a natural generalization of the Carath\'eodory condition in Definition \ref{cara1}.

\begin{definition}\label{cara2}
Let $I$ be a contractive operator-valued map on $\D^2$. Then $\tau \in \T^2$ is a carapoint for $I$ if there exists a sequence $\{\la_n\} \subset \D^2$ tending to $\tau$ such that 
$$ \liminf_{\la \to \tau} \frac{1 - \norm{I(\la)}}{1 - \norm{\la}_\infty} \text{ is bounded.}$$
\end{definition}  

The following Lemma is proved in \cite[Theorem 3.6]{aty12}.

\begin{lemma}\label{iprop}
Let $Y$ be a positive contraction on a Hilbert space $\M$ and let $\tau \in \T^2$. Define an operator-valued, degree $(1,1)$ rational map $I_Y(\la)$ from $\C^2 \to \Lop (\M)$ by
\beq \label{olddefi}
 I_Y(\la) = \frac{\cc\tau^1\la^1 Y + \cc\tau^2\la^2 (1 - Y) - \cc\tau^1\cc\tau^2\la^1\la^2}{1 - \cc\tau^1\la^1(1 - Y) - \cc\tau^2\la^2 Y}.
\eeq

Then $I_Y$ is contractive and analytic on $\D^2$, $\tau$ is a singular carapoint for $I_Y$ (in the sense of Definition \ref{cara2}), and $I_Y(\tau) = 1_\M$.
\end{lemma}

Note that a a generalized model reduces to a standard Hilbert space model in the case that the operator $Y$ is a projection. 

The utility of generalized models at carapoints arises from the existence of a model for which the model function $v$ extends continuously to $\tau$ on sets that approach $\tau$ nontangentially. 

\begin{theorem}[Agler, Tully-Doyle, Young] \label{genmodexists}
Let $\tau \in \T^2$ be a carapoint for $\vp \in \schur$. Then there exists a Hilbert space $\M$, a positive contraction $Y$ on $\M$, an analytic map $v:\D^2 \to \M$ such that for all $\la, \mu \in \D^2$,
	$$ 1 - \cc{\ph(\mu)}\ph(\la) = \ip{(1 - I(\mu)\ad I(\la))v_\la}{v_\mu}$$
and $\tau$ is a $C$-point for $(\M, v, I_Y)$.

\end{theorem}

We begin by characterizing the directional derivative of a function $\ph$ in terms of the positive contraction $Y$. The following lemma appears in the proof of Theorem 4.1 of \cite{aty12}.

\begin{lemma} \label{dderiv}
If $\ph$ has a carapoint at $\tau \in \T^2$ then there exist a Hilbert space $\M$, a positive contraction $Y$ on $\M$ and a vector $v_\tau \in \M$ such that the directional derivative of $\ph$ for a direction $\delta$ pointing into the bidisc at $\tau$ is given by the formula
\[
D_\delta \ph(\tau) = \ip{\frac{\cc\tau^1\cc\tau^2\delta^1\delta^2}{\cc\tau_1\delta^1(1-Y) + \cc\tau^2\delta^2 Y}v_\tau}{v_\tau}.
\]
\end{lemma}
\begin{proof}
Let $\la_t = \tau + t\delta$ where $\delta = (\delta^1, \delta^2) \in \C^2$ and $\RE \delta^1, \RE \delta^2 < 0$ (so that $\la_t \in \D^2$ for small enough $t>0$.) 

By Theorem \ref{genmodexists}, $\ph$ has a generalized model such that 
	\beq\label{genmodeeq} 1 - \ph(\la)\cc\ph({\mu}) = \ip{(1 - I_Y(\mu)\ad I_Y(\la))v_\la}{v_\mu}, \eeq
and such that $\tau$ is a $C$-point for the model. $v_\la$ extends continuously to the boundary on nontangential sets approaching $\tau$, and thus has a nontangential limit $v_\tau$ as $\la \to \tau$.  Then applying limits to \eqref{genmodeeq} as $\mu \nt \tau$ gives
\[
1 - \cc{\ph(\tau)}\ph(\la) = \ip{(1 - I(\tau)\ad I(\la))v_\la}{v_\tau}.
\]
Multiplying through by $-\ph(\tau)$ gives
\begin{align}
\ph(\la) - \ph(\tau) &= \ph(\tau)\ip{(I(\la) - 1)v_\la}{v_\tau} \notag \\
&= \ph(\tau)\ip{(I(\la) - 1)v_\tau}{v_\tau} + \ph(\tau)\ip{(I(\la) - 1)(v_\la - v_\tau)}{v_\tau} \label{diffph}.
\end{align}

 The difference $I(\la_t) - I(\tau)$ is given by 
\begin{align}
&I(\la_t) - I(\tau) =  \left[\frac{\cc\tau^1\la_t^1Y + \cc\tau^2\la_t^2(1-Y) - \cc\tau^1\cc\tau^2\la_t^1\la_t^2}{1 - \cc\tau^1\la_t^1(1-Y) + \cc\tau^2\la_t^2 Y} - 1\right] \notag \\
&= \left[  \frac{\cc\tau^1(\tau^1 + t\delta^1)Y + \cc\tau^2(\tau^2 + t\delta^2)(1-Y) - \cc\tau^1\cc\tau^2(\tau^1 +t\delta^1)(\tau^2+t\delta^2)}{1 - \cc\tau^1(\tau^1 + t\delta^1)(1-Y) - \cc\tau^2(\tau^2 + t\delta^2)Y} - 1 \right] \notag \\
&= \left[  \frac{(1 + t\cc\tau^1\delta^1)Y + (1 + t\cc\tau^2\delta^2)(1-Y) - (1 +t\cc\tau^1\delta^1)(1+t\cc\tau^2\delta^2)}{1 - (1 + t\cc\tau^1\delta^1)(1-Y) - (1 + t\cc\tau^2\delta^2)Y} - 1 \right] \notag \\
&= \frac{t\cc\tau^1\cc\tau^2\delta^1\delta^2}{\cc\tau^1\delta^1(1-Y) + \cc\tau^2\delta^2 Y}, \label{Idiff}
\end{align}
(We have used the fact that $I_Y(\tau) = 1_\M$ from Lemma \ref{iprop}).
Upon dividing by $t$ and applying the limit as $t \to 0^+$, we get
\begin{align} 
D_\delta I(\tau) &= \frac{\cc\tau^1\cc\tau^2\delta^1\delta^2}{\cc\tau^1\delta^1(1-Y) + \cc\tau^2\delta^2Y} \\
&= \frac{\delta^1\delta^2}{\tau^2\delta^1(1-Y) + \tau^1\delta^2Y}\label{ddI}.
\end{align}
Combining with \eqref{diffph}, we calculate a difference quotient.
\begin{align*}
\frac{\ph(\la_t) - \ph(\tau)}{t} &= \ph(\tau)\frac{1}{t}\ip{(I(\la_t) - 1)v_{\la_t}}{v_\tau} \notag \\
&= \ph(\tau)\ip{\frac{I(\la_t) - I(\tau)}{t}v_\tau}{v_\tau} \\ &\hspace{.2in} + \ph(\tau)\ip{\frac{I(\la_t) - I(\tau)}{t}(v_{\la_t} - v_\tau)}{v_\tau}. 
\end{align*}
Finally, letting $t \to 0^+$, we conclude
\beq \label{ddphi}
D_\delta \ph(\tau) = \ip{\frac{\delta^1\delta^2}{\tau^2\delta^1(1-Y) + \tau^1\delta^2 Y}v_\tau}{v_\tau}.
\eeq
\end{proof}

(A similar argument appears in \cite{amy10a} in the proof of Lemma 4.2.)

\section{Structure of rational model functions}\label{sectioni}

By Theorem \ref{genmodexists}, any Schur function $\ph$ with a carapoint at $\tau \in \T^2$ has a continuous generalized model at $\tau$. Be removing the modeling of a discontinuity from $v_\la$, we lose the ability to characterize the nature of the discontinuity in terms of the model; that is, we cannot use the behavior of $v_\la$ to examine differential structure of $\ph$ at $\tau$. Our main objective is to recapture a geometric condition that distinguishes between these two cases, in the spirit of the two variable Julia-Carath\'eodory theorem in \cite{amy10a}.

We begin with an example of a family of simple rational functions that possess a single nondifferentiable carapoint, illustrating the complicated nature of boundary singularities even for nice functions.

\begin{lemma}\label{phiy}
Let
	$$\ph_y(\la) = \frac{\cc\tau^1\la^1 y + \cc\tau^2\la^2 (1-y) - \cc\tau^1\cc\tau^2\la^1\la^2}{1 - \cc\tau^1\la^1(1-y) - \tau^2\la^2 y}.$$
For all $y \in (0,1)$, the function $\ph_y$ has a nondifferentiable carapoint at the point $\tau = (\tau^1, \tau^2) \in \T^2$.
\end{lemma}
\begin{proof}
By calculation,
	$$D_{-\delta}\ph_y(\la) = \frac{ \delta^1 \delta^2}{\tau^2\delta^1 (1 - y) + \tau^1\delta^2 y},$$
which is not linear in $\delta$, and so $\ph_y$ fails to be nontangentially differentiable at $\tau$. To see that $\ph_y$ has a carapoint at $\tau$, it is enough to check the Carath\'eodory condition along the ray $(r\tau_1, r\tau^2)$ as $r \to 1$. On this ray,
	$$\ph_y(r\tau^1, r\tau^2) = r.$$
Hence, if $\la = (r\tau^1,r\tau^2)$ tends to $\tau$,
		\begin{align*}
		\liminf_{\la \to \tau} \frac{1 - \abs{\ph_y(\la)}}{1 - \norm{\la}_\infty} &= \liminf_{r\to 1} \frac{1 - \abs{r}}{1 - r} \\
		&= 1
		\end{align*}
and so $\ph_y$ has a carapoint at $\tau$.
\end{proof}

In the boundary cases $y = 0$ and $y = 1$, the functions $\ph_0$ and $\ph_1$ are well behaved, as 
\beq\label{boundingcases} 
	\ph_1(\la) = \cc\tau^1\la^1, \hspace{.5in} \ph_0(\la) = \cc\tau^2\la^2,
\eeq 
respectively. That is, the singularity at $\tau$ disappears.

Note that the function $\ph_y$ is the scalar case of the generalized model function $I_Y$ in Theorem \ref{genmodexists}.

\begin{lemma}\label{iyint}
Let $Y$ be a positive contraction on a Hilbert space $\M$. Then there exists a projection-valued measure $E$ supported on the unit interval such that 
\beq\label{int}
I_Y(\la) = \int \! \ph_y(\la) \, \dd E(y).
\eeq

Furthermore, if $\sigma(Y) \cap (0,1) = \emptyset$, that is $Y = P$ is a projection, then 
\beq \label{secret}
I_Y(\la) = \cc\tau^1 \la^1 P + \cc\tau^2 \la^2 (1-P),
\eeq
the operator present in the standard model \eqref{model2}.
\end{lemma}

\begin{proof}
Equation \eqref{int} follows immediately on application of the spectral theorem to $Y$.

To see Equation \eqref{secret}, note that if $\sigma(Y)\cap(0,1) = \emptyset$, so that $Y = P$, then 
$$\int \ph_y(\la) \, \dd E(y) = \cc\tau^1 \la^1 E_1 + \cc \tau^2 \la^2 E_0 = \cc\tau^1 \la^1 P + \cc\tau^2 \la^2 (1-P).$$ 
\end{proof}

\begin{lemma}\label{refereefix} Let $(\M, v, I_Y)$ be a model for $\ph\in\schur$. If $Y$ is a projection, then $(\M, v, I_Y)$ is a standard model.
\end{lemma}

\begin{proof}
If $Y$ is a projection, then $I_Y$ can be written as in \eqref{secret}. Then the generalized model equation can be rewritten as

\begin{align*} 
1 - \cc\ph(\mu)\ph(\la) &= \ip{1 - I(\mu)\ad I(\la) u_\la}{u_\mu} \\
&= \ip{(1 - (\cc\tau^1\mu^1 P + \cc\tau^2\mu^2 (1-P))\ad(\cc\tau^1\la^1 P + \cc\tau^2\la^2 (1-P)))u_\la}{u_\mu} \\
&= \ip{(1 - \mu\ad\la)u_\la}{u_\mu},
\end{align*}
and thus $(\M, v, I_Y)$ is a standard model as in Definition \ref{model}.
\end{proof}

To investigate the behavior of a generalized model $(\M, u, I_Y)$, we first develop some properties of the one parameter family of scalar functions $\ph_y$. Every function $\ph_y$ has an explicit model (a statement that appears without proof as Proposition 6.3 in \cite{amy10a}).

\begin{lemma}\label{scalarphi}
For a real number $y$, $0<y<1$, let $\ph_y$ be the inner function on $\C^2$ given by
 \beq\label{phiy2}
 \ph_y(\la) = \frac{\cc\tau^1 \la^1 y + \cc \tau^2\la^2 (1-y) - \cc\tau^1\cc\tau^2\la^1\la^2}{1 - \cc\tau^1 \la^1 (1-y)  - \cc\tau^2\la^2 y}. 
 \eeq
 Then any model $(\M,u)$ of $\ph_y$ has a $B$-point at $\tau = (\tau^1, \tau^2) \in \T^2$. Furthermore, $(\C^2, u_y)$ is a model for $\ph_y$, where $u_{y,\la}$ has the form
 \beq \label{imodel}
 u_{y,\la} = \frac{1}{1 - \cc\tau^1\la^1(1-y)- \cc\tau^2\la^2 y} \vectwo{\sqrt{y}(1 - \cc\tau^2\la^2)}{\sqrt{1-y}(1 - \cc\tau^1\la^1)}.
 \eeq

With respect to the orthonormal basis of $\C^2$ given by
\[
e_+ = \vectwo{\sqrt{1-y}}{\sqrt{y}}, \hspace{.1in} e_-=\vectwo{\sqrt{y}}{-\sqrt{1-y}},
\]
we can write the model as 
\beq \label{scalarmodel}
u_{y,\la} = \frac{\sqrt{(1-y)y}(\cc\tau^1\la^1 - \cc\tau^2\la^2)}{1 - \cc\tau^1\la^1 (1-y) - \cc\tau^2\la^2 y}e_+ + e_-.
\eeq

\end{lemma}

\begin{proof}
A straightforward calculation shows that 
\[
1 - \ph_y(\la)\ad \ph_y(\la) = \ip{(1 - \mu\ad\la)u_{y,\la}}{u_{y,\mu}}.
\]

To show that $\tau$ is a $B$-point for $\ph_t$, we need to show that $u_{y,\la}$ is bounded as $\la \to \tau$ nontangentially. Let $S$ be a set in $\D^2$ that approaches $\tau$ nontangentially. Then there exists a $c > 0$ so that for $\la \in S$,
\[
\abs{\tau - \la} \leq c (1 - \abs{\la}).
\]
We will show that the coefficient of $e_+$ in \eqref{scalarmodel} is bounded on $S$. To do so, notice that 
\begin{align*}
\abs{\cc\tau^1\la^1 - \cc\tau^2\la^2} &= \abs{(1 - \cc\tau^2\la^2) + (\cc\tau^1\la^1 - 1)} \\
&\leq \abs{1-\cc\tau^1\la^1} + \abs{1 - \cc\tau^2\la^2} \\
&\leq 2\max \{\abs{1-\cc\tau^1\la^1}, \abs{1 - \cc\tau^2\la^2}\} \\
&\leq 2c\min \{(1 - \abs{\cc\tau^1\la^1}),(1 - \abs{\cc\tau^2\la^2})\} \\
&\leq 2c [(1-y)(1 - \abs{\cc\tau^1\la^1}) + y(1 - \abs{\cc\tau^2\la^2})] \\
&= 2c [(1-y) - (1-y)\abs{\cc\tau^1\la^1} + y - y\abs{\cc\tau^2\la^2}] \\
&= 2c [1 - (1-y)\abs{\cc\tau^1\la^1} - y\abs{\cc\tau^2\la^2}] \\
&\leq 2c\abs{1 - (1-y)\cc\tau^1\la^1 - y\cc\tau^2\la^2}.
\end{align*}
Then $u_{y,\la}$ is bounded on the set $S$, as
\begin{align}
\norm{u_{y,\la}} &= \norm{ \frac{\sqrt{y(1-y)}(\cc\tau^1\la^1 - \cc\tau^2\la^2)}{1 - (1-y)\cc\tau^1\la^1 - y\cc\tau^2\la^2}e_+ + e_-} \notag \\
 &\leq 2c \sqrt{y(1-y)} \norm{e_+} + \norm{e_-} \notag \\  &= 2c\sqrt{y(1-y)} + 1, \label{ubound}
\end{align}
which depends only on $y$. Then $u_{y,\la}$ is bounded as $\la \to \tau$ nontangentially, and so by Theorem \ref{modelbpoint}, $\tau$ is a $B$-point for $\ph_y$. 
\end{proof}

Together, Lemmas \ref{phiy} and \ref{scalarphi} imply that any model for $\ph_{y}$ where $y \in (0,1)$ has a $B$-point that is not a $C$-point at $\tau$.

\section{The structure of generalized models at a carapoint}

We are now prepared to examine the relationship between the geometry of the model $(\M, u, I_Y)$ and the differentiability of $\ph$. We begin by addressing the trivial case in which $\ph \in \schur$ has a generalized model where the contraction $Y$ in the formula for $I$ is in fact a projection.

\begin{lemma}\label{proj}
Suppose that $\ph \in \schur$ has a continuous generalized model $(\M, v, I_P)$ where $P$ is a projection acting on $\M$. Then $\ph$ has a differentiable carapoint at $\tau$.
\end{lemma}
\begin{proof}
By Lemma \ref{refereefix}, $(\M, v, I_P)$ is a standard model. By hypothesis, $v$ extends continuously at $\tau$, and so $\tau$ is a $C$-point for $(\M, v, I_P)$ viewed as a standard model. Therefore, by Theorem \ref{modelcpoint}, $\ph$ has a differentiable carapoint at $\tau$.
\end{proof}

We need the following geometrical Lemma about the behavior of the model function at $\tau$. Recall that if a model has a $C$-point at $\tau$ then the model function extends continuously to $\tau$ on sets approaching $\tau$ nontangentially (see Definition \ref{modelpoints}). In this case, a sequence $v_\la$ as $\la \to \tau$ will have a nontangential limit at $\tau$, which we denote $\lim_{\la\nt\tau} v_\la = v_\tau$.
\begin{theorem}\label{boundedvchi}
Let $\ph \in \schur$ have a carapoint at $\tau \in \T^2$. Then for a generalized model $(\M, v, I_Y)$ with a $C$-point at $\tau$,
\[
\norm{v_\tau} > 0.
\] 
\end{theorem}
\begin{proof}
On taking limits as $\mu \to \la$, the model equation 
\[
1 - \cc{\ph(\mu)}\ph(\la) = \ip{(1 - I(\mu)\ad I(\la))v_\la}{v_\mu}
\]
becomes
\beq\label{posv}
1 - \norm{\ph(\la)}^2 = \norm{v_\la}^2 - \norm{I(\la)v_\la}^2.
\eeq
From \eqref{Idiff}, when $\la_t = \tau + t\delta$,
	$$ I(\la_t) - I(\tau) = \frac{t\delta^1\delta^2}{\tau^2\delta^1(1-Y) + \tau^1\delta^2 Y}.$$
When $\la_t = \tau + t(-\tau)$, this becomes 
\[
I(\la_t) - 1 = -t
\]
and so $I(\la_t) = 1 - t$. Plugging into \eqref{posv}, 
\beq \label{numjulia}
1 - \abs{\ph(\la_t)}^2 = \norm{v_{\la_t}}^2 - \norm{(1-t)v_{\la_t}}^2 = (2t - t^2)\norm{v_{\la_t}}^2.
\eeq
Additionally,
\beq \label{denomjulia}
1 - \norm{\la_t}^2_\infty = 1 - \norm{(\tau + t(-\tau)}^2_\infty = (2t - t^2)\norm{\tau}_\infty^2 = (2t - t^2).
\eeq
Combining \eqref{numjulia} with \eqref{denomjulia} yields
\[
 \norm{v_{\la_t}}^2 = \frac{1 - \abs{\ph(\la_t)}^2}{1 - \norm{\la_t}^2_\infty}.
\]
On application of limits, we get
\[
\norm{v_\tau}^2 = \lim_{t \to 0^+} \frac{1 - \abs{\ph(\la_t)}^2}{1 - \norm{\la_t}^2_\infty} = \lim_{t \to 0^+} \frac{1 - \abs{\ph(\la_t)}}{1 - \norm{\la_t}_\infty}.
\]
However, as $\tau$ is a carapoint of $\ph$, this quantity has a non-tangential limit, and hence
\[
\lim_{t \to 0^+} \frac{1 - \abs{\ph(\la_t)}}{1 - \norm{\la_t}_\infty} = \liminf_{\la \nt \tau} \frac{1 - \abs{\ph(\la)}}{1 - \norm{\la}_\infty}.
\]
(see, for example, \cite{amy10a} or \cite{jaf93}). 
Finally, so long as $\ph$ is not constant, as $\tau$ is a carapoint for $\ph$, by \cite[Theorem 4.9]{amy10a}, 
\[
\liminf_{\la \nt \tau} \frac{1 - \abs{\ph(\la)}}{1 - \norm{\la}_\infty} = \alpha > 0,
\]
which gives $\norm{v_\tau} > 0$ .
\end{proof}

We are now prepared to state and prove the a converse of Theorem \ref{genmodexists}. The content of the following Theorem is essentially that a generalized model is a sum of scalar standard models, and the singular behavior modeled by $I_Y$ is built up from the singular behavior of the scalar functions $\ph_y$ sharing a singular carapoint at $\tau$.

\begin{theorem}\label{forward}
Let $\ph$ be a function in $\schur$ and $(\M, v, I_Y)$ a model for $\ph$ with a $C$-point at $\tau$. Then $\tau$ is a carapoint for $\ph$.
\end{theorem}

\begin{proof}

First, in the case that $\sigma(Y) \cap (0,1) = \emptyset$, $Y$ is a projection and Lemma $\ref{proj}$ implies that $\ph$ has a differentiable carapoint at $\tau$.
 
On the other hand, assume that $\sigma(Y) \cap (0,1) \neq \emptyset$. By Lemma \ref{iyint}, there exists a spectral measure $E$ such that
\[
I_Y(\la) = \la^1 E_1 + \la^2 E_0 + \int_{(0,1)} \! \ph_{y}(\la) \, \dd  E(y),
\]

As $(\M, v, I_Y)$ is a model,
\beq \label{modelcalc0}
1 - \cc{\ph(\mu)}\ph(\la) = \ip{(1 - I_Y(\mu)\ad I_Y(\la))v_\la}{v_\mu}.
\eeq
We will show that $\tau$ is a carapoint for $\ph$ by deriving a standard model for $\ph$ and then proving that the model is nontangentially bounded at $\tau$, that is we will show that $\tau$ is a $B$-point and thus by Theorem \ref{bpoint} that $\tau$ is a carapoint for $\ph$. 

First, we derive an expression for $1 - I_Y(\mu)\ad I_Y(\la)$: 
\begin{align}
&1 - I_Y(\mu)\ad I(\la) \\ &=  1 - \left(\mu^1 E_1 +\mu^2 E_2 + \int_{(0,1)} \! \ph_y(\mu)\,\dd E(y)\right)\ad \times \\&\hspace{1in}\left(\la^1 E_1 + \la^2 E_0 + \int_{(0,1)}\! \ph_y(\la)\,\dd E(y)\right) \notag \\
&= 1 - \left( \cc\mu^1\la^1 E_1 + \cc\mu^2\la^2 E_0 + \int_{(0,1)}\! \cc{\ph_{y}(\mu)}\ph_{y}(\la)\,\dd E(y) \right) \notag \\
&=  (1 - \cc\mu^1\la^1)E_1 + (1 - \cc\mu^2\la^2)E_0 \\ &\hspace{1in}+ \int_{(0,1)}\!(1 - \cc{\ph_{y}(\mu)}\ph_{y}(\la))\,\dd E(y). \label{modelcalc1}
\end{align}
Each function $\ph_y$ can be modeled with $(\C^2, u_{y,\la})$ as given in Lemma \ref{scalarphi}, so continuing from \eqref{modelcalc1}, we get
\begin{align}
&\phantom{= \text{ }} (1 - \cc\mu^1\la^1)E_1 + (1 - \cc\mu^2\la^2)E_0 + \int_{(0,1)}\!(1 - \cc{\ph_{y}(\mu)}\ph_{y}(\la))\,\dd E(y) \notag \\
&=  (1 - \cc\mu^1\la^1)E_1 + (1 - \cc\mu^2\la^2)E_0 + \int_{(0,1)}\! \ip{(1- \mu\ad\la)u_{y, \la}}{u_{y, \mu}}\,\dd E(y) \notag \\
&= (1 - \cc\mu^1\la^1)E_1 + (1 - \cc\mu^2\la^2)E_0 \notag\\ &\hspace{.5in} + \int_{(0,1)}\! \ip{(1- \cc\mu^1\la^1)u^1_{y, \la}}{u^1_{y, \mu}}\, \dd E(y) \notag \\&\hspace{1in}+ \int_{(0,1)}\! \ip{(1- \cc\mu^2\la^2)u^2_{y, \la}}{u^2_{y, \mu}}\, \dd E(y) \notag \\
&= (1 - \cc\mu^1\la^1)\left( E_1 + \int_{(0,1)}\! \ip{u^1_{y, \la}}{u^1_{y, \mu}}\, \dd E(t) \right) \notag \\&\hspace{.5in}+ (1 - \cc\mu^2\la^2)\left(E_0 + \int_{(0,1)}\! \ip{u^2_{y, \la}}{u^2_{y, \mu}}\, \dd E(y)\right). \label{modelcalc2}
\end{align}
If we let
\begin{align}
&U_1(\la) = 1E_1 + 0E_0 + \int_{(0,1)}\! u_{y, \la}^1 \, \dd E(y), \notag \\ &U_2(\la) = 0E_1 + 1E_0 + \int_{(0,1)}\! u_{y, \la}^2\,\dd E(y) \label{umodel}
\end{align}
then we can substitute into \eqref{modelcalc2} to get
\beq
1 - I(\mu)\ad I(\la) = (1 - \cc\mu^1\la^1)U^1(\mu)\ad U^1(\la) + (1 - \cc\mu^2\la^2)U^2(\mu)\ad U^2(\la).
\eeq
Upon substitution of this expression into the generalized model equation \eqref{modelcalc0}, we get
\begin{align*}
1 - \cc{\ph(\mu)}\ph(\la) &= \ip{(1 - I(\mu)\ad I(\la))v_\la}{v_\mu} \\
&= \ip{((1 - \cc\mu^1\la^1)U_1(\mu)\ad U_1(\la) + (1 - \cc\mu^2\la^2)U_2(\mu)\ad U_2(\la))v_\la}{v_\mu} \\
&= (1 - \cc\mu^1\la^1)\ip{U_1(\la)v_\la}{U_1(\mu)v_\mu} + (1 - \cc\mu^2\la^2)\ip{U_2(\la)v_\la}{U_2(\mu)v_\mu}. 
\end{align*}
Then we have shown that $(\M \oplus \M, U)$ is a model for $\ph$ in the sense of Definition \ref{model}, where $U_\la$ is the function
\beq\label{stdmod}
U_\la = U(\la) = \vectwo{U_1(\la)v_\la}{U_2(\la)v_\la}.
\eeq
To show that $\tau$ is a $B$-point for $\ph$, by Theorem \ref{modelbpoint} it is enough to show that $U(\la)$ is bounded as $\la \nt \tau$. We will show that the component $U_1(\la)v_\la$ is bounded on a set that approaches $\tau$ nontangentially (that $U_2(\la)v_\la$ is bounded follows similarly). First, $U_1(\la)$ is a bounded operator. To see this, let $S$ be a set that approaches $\tau$ nontangentially such that for all $\la \in S$,
	$$\abs{\tau - \la} \leq c(1-\abs{\la}).$$ 

Trivially, the operator $1E_1$ is bounded. By Lemma \ref{scalarphi}, for any $y$ with $0<y<1$, for all $\la \in S$, 
\[
\norm{u_{y,\la}} \leq 2c\sqrt{y(1-y)} + 1.
\]
As the maximum of the function $f(x) = \sqrt{y(1-y)}$ is $1/2$, for all $y \in (0,1)$,
\[
\norm{u_{y,\la}} \leq c + 1.
\]
Thus, the family $\{u_{y,\la}\}$ is uniformly bounded on $S$. Let $u, v$ be arbitrary vectors in $\M$. Since $E$ is a spectral measure, 
\begin{align}
\abs{\ip{\left(\int_{(0,1)} \! u^i_{y,\la} \, \dd E(y)\right)u}{v}}  &= \abs{\int_{(0,1)} \! u^i_{y,\la} \, \dd E_{u,v}(y)} \notag \\
&\leq \int_{(0,1)} \! \abs{u^i_{y,\la}} \, \dd \abs{E_{u,v}(y)} \notag \\
&\leq \int_{(0,1)} \! (c+1) \, \dd \abs{E_{u,v}(y)} \notag \\
&\leq (c+1)\norm{E_{u,v}(y)} \notag \\
&\leq (c+1)\norm u \norm v. \label{Ubound}
\end{align}
Then $U_1(\la)$ is a bounded operator, and as the bound does not depend on the choice of $\la \in S$, the family of operators $\{U_1(\la)\}_{\la \in S}$ is uniformly bounded on $S$.  By \eqref{Ubound}, for all $\la \in S$,
\beq \label{Uvest}
\norm{U^i(\la)v_\la} \leq \norm{U^i(\la)}\norm{v_\la} \leq \sqrt{c+1}\norm{v_\la}.\eeq
 Recall that by hypothesis, the generalized model function $v_\la$ has a $C$-point at $\tau$ and thus $v_\la \to v_\tau$ as $\la \nt \tau$. Then for any sequence $\la_n \nt \tau$ in $S$, by Theorem \ref{boundedvchi}, 
  $$ \norm{U^i(\la)v_\la} \leq \sqrt{c+1}\norm{v_\la} \to \sqrt{c+1}\norm{v_\tau} = (\sqrt{c+1})\alpha.$$  

As each component of the model function $U$ is bounded on $S$ as $\la \to \tau$, so too is $U$. Then the model $(\M \oplus \M, U)$ has a $B$-point at $\tau$, and thus $\ph$ has a carapoint at $\tau$ by Theorem \ref{modelbpoint}. 
\end{proof}

\section{Model geometry and differentiability}

We are now in position to establish a condition on a generalized model for a function $\ph \in \schur$ at a carapoint $\tau$ that characterizes the differential structure of $\ph$ at $\tau$, in keeping with the spirit of Agler, {\mc}Carthy, and Young's generalization of the Julia-Carath\'eodory Theorem \cite{amy10a}. Recall that the purpose of a generalized model is to move the singular behavior out of the model function $v_\la$ and into the operator-valued map $I_Y(\la)$. Accordingly, while we can no longer look at the behavior of the model function $v_\la$ to characterize the differentiability of $\ph$, the positive contraction $Y$ encodes this information.

\begin{definition}
Suppose that a Schur function $\ph$ with a carapoint at $\tau$ has a generalized model $(\M, v, I_Y)$. Let $\N = \ker Y(1-Y)$ and denote the orthogonal complement of $\N$ in $\M$ by $\N^\perp$. Say that a generalized model is \emph{regular} if $P_{\N^\perp} v_\tau = 0$. Otherwise, the model is \emph{singular}. If instead $P_{\N} v_\tau = 0$, then the generalized model is \emph{purely singular}.
\end{definition}

\begin{remark}
We should point out that by the above definitions, if $Y$ is a projection then $(\M, v, I_Y)$ is a regular generalized model.
\end{remark}

These definitions allow us to make an explicit classification of the nontangential differentiability of a function $\ph$ at a carapoint $\tau$ in terms of the geometry of the model.

\begin{theorem}\label{forward2}
Let $\ph$ be a function in $\schur$.  $\ph$ has a singular generalized model at $\tau$ if and only if  $\ph$ has a nondifferentiable carapoint at $\tau$.
\end{theorem}

\begin{proof}
$(\Rightarrow)$: We show the contrapositive. Suppose that $\ph$ has a nontangentially differentiable carapoint at $\tau$. Let $(\M \oplus \M, U)$ be the standard model derived from $(\M, u, I_Y)$ given in \eqref{umodel} and \eqref{stdmod}. Then by Theorem \ref{modelcpoint}, the model function $U(\la)$ extends by continuity to $\tau$ on any set $S$ that approaches $\tau$ nontangentially, and so there exists a vector $U(\tau)$ so that
\[
\lim_{\la \nt \tau} U(\la) = U(\tau).
\]

Note that 
\[
\vectwo{U_1(\la)v_\la}{U_2(\la)v_\la} = \vectwo{U_1(\la) v_\tau}{U_2(\la) v_\tau} + \vectwo{U_1(\la) (v_\la - v_\tau)}{U_2(\la) (v_\la - v_\tau)},
\]
and so
$$ U(\tau) = \lim_{\la \nt \tau} U(\la) = \lim_{\la \nt \tau} \vectwo{U_1(\la)v_\tau}{U_2(\la)v_\tau}. $$

Now, consider the quantity 
\begin{align*}
&\norm{\vectwo{U_1(\la) v_\tau}{U_2(\la) v_\tau} - \vectwo{U_1(\mu) v_\tau}{U_2(\mu) v_\tau}}^2 \\ &= \norm{ \vectwo{(U_1(\la) - U_1(\mu))v_\tau}{(U_2(\la) - U_2(\mu))v_\tau}}^2 \\
&= \sum_{i=1}^2 \ip{ (U_i(\la) - U_i(\mu)) v_\tau}{(U_i(\la) - U_i(\mu)) v_\tau} \\
&= \sum_{i=1}^2 \ip{ (U_i(\la) - U_i(\mu))\ad(U_i(\la) - U_i(\mu)) v_\tau}{ v_\tau} \\
&= \sum_{i=1}^2 \ip{\int_{(0,1)}\!\abs{u^i_{y,\la} - u^i_{y,\mu}}^2\,\dd E(y) v_\tau}{v_\tau} \\
&= \sum_{i=1}^2 \int\!\abs{u^i_{y,\la} - u^i_{y,\mu}}^2\,\dd E_{v_\tau,v_\tau}(y).
\end{align*}

For any distinct sequences $\la_n, \mu_n \nt \tau$, 
\begin{align} 
&\lim_{n\to\infty} \sum_{i=1}^2 \int\!\abs{u^i_{y,\la_n} - u^i_{y,\mu_n}}^2\,\dd E_{v_\tau,v_\tau}(y) \notag \\ &= \lim_{n\to\infty} \norm{\vectwo{U_1(\la_n) v_\tau}{U_2(\la_n) v_\tau} - \vectwo{U_1(\mu_n) v_\tau}{U_2(\mu_n) v_\tau}}^2 = 0. \label{contlimit}
\end{align}
By Theorem \ref{boundedvchi},  $\norm{v_\tau} > 0$, and so $E_{v_\tau, v_\tau}$ is a finite, positive measure supported on $\sigma(Y)$ (see, e.g. \cite[p.257]{con97}).  Then for $y \in \sigma(Y)\cap (0,1)$, Equation \eqref{contlimit} implies that
\beq \label{badcpoint}
\lim_{n\to\infty} \abs{u^i_{y,\la_n} - u^i_{y,\mu_n}} = 0.
\eeq
But this would imply that the model function $u_{y, \la}$ had a $C$-point at $\tau$, which cannot happen for $y \in (0,1)$ by Lemma \ref{phiy} and Lemma \ref{scalarphi}. Thus if $U$ extends continuously at $\tau$, it must be the case that $P_{\ker Y(1-Y)^\perp} v_\tau = 0.$ We conclude that the generalized model $(\M, v, I_Y)$ cannot be singular.

\medskip

$(\Leftarrow)$: Suppose that $\ph \in \schur$ has a nondifferentiable carapoint at $\tau$. By Theorem \ref{genmodexists}, there exists a generalized model $(\M, v, I_Y)$ with a $C$-point at $\tau$.

To show that $(\M, v, I_Y)$ is singular, we show that 
	$$P_{\N^\perp}v_{\tau} \neq 0,$$
 using facts about the directional derivative of $\ph$ at $\tau$. From Lemma \ref{dderiv}, for $\delta$ pointing into the bidisk,
\beq \label{undecomp}
D_\delta \ph(\tau) = \ip{\frac{\delta^1\delta^2}{\tau^2\delta^1 (1-Y) + \tau^1\delta^2 Y}v_\tau}{v_\tau}.
\eeq

Decompose $Y$ as $1E_1 + 0E_0 + Y_0$, where $E_1$ and $E_0$ are projections onto $\ker Y$ and $\ker 1-Y$ respectively. Let $E = 1 - E_0 - E_1$. Then $Y$ can be written in block matrix form as 
\[
Y = \bbm 1 & 0 & 0 \\ 0 & 0 & 0 \\ 0 & 0 & Y_0 \ebm \begin{array}{c} \M_1 \\ \M_0 \\ \M_s \end{array}
\]
where $\M_1 = E_1 \M, \M_0 = E_0 \M$, and $\M_s = E \M$. (Recall that $Y$ is a positive contraction.)
Then
\begin{align*}
(\tau^2\delta^1(1-Y) + \tau^1\delta^2(Y))\inv &= \bbm \tau^1\delta^2 & 0 & 0 \\ 0 & \tau^2\delta^1 & 0 \\ 0 & 0 & \tau^2\delta^1(1-Y_0) + \tau^1\delta^2 Y_0 \ebm \inv \\
&= \bbm \frac{\cc\tau^1}{\delta^2} & 0 & 0 \\ 0 & \frac{\cc\tau^2}{\delta^1} & 0 \\ 0 & 0 & (\tau^2\delta^1(1-Y_0) + \tau^1\delta^2 Y_0)\inv \ebm,
\end{align*}
and so
\begin{align*}
\frac{\delta^1\delta^2}{\tau^2\delta^1 (1 - Y) + \tau^1\delta^2 Y} &= \bbm \cc\tau^1\delta^1 & 0 & 0 \\ 0 & \cc\tau^2\delta^2 & 0 \\ 0 & 0 & \frac{\delta^1\delta^2}{\tau^2\delta^1 (1 - Y_0) + \tau^1\delta^2 Y_0} \ebm \\ &= \cc\tau^1\delta^1 E_1 + \cc\tau^2\delta^2 E_0 + \frac{\delta^1\delta^2}{\tau^2\delta^1 (1 - Y_0) + \tau^1\delta^2 Y_0} E.
\end{align*}
Then the formula given in \eqref{undecomp} decomposes as
\begin{align}
D_\delta \ph(\tau) = &\ip{\cc\tau^1\delta^1 E_1 v_\tau}{E_1 v_\tau} + \ip{\cc\tau^2\delta^2 E_0 v_\tau}{E_0 v_\tau} \notag \\
&+  \ip{\frac{\delta^1\delta^2}{\tau^2\delta^1 (1-Y_0) + \tau^1\delta^2 Y_0} E v_\tau}{E v_\tau}. \label{decompderiv}
\end{align}

As $\ph$ has a nondifferentiable carapoint at $\tau$, the directional derivative cannot be linear in $\delta$. This implies that $Ev_\tau$ must be non-zero, but this is precisely the condition
\[
\lim_{\la \nt \tau} P_{\N^\perp} v_\la \neq 0.
\]
Therefore, $(M, v, I_Y)$ is a singular generalized model for $\ph$ at $\tau$.
\end{proof}

\begin{theorem}\label{diffmain}
Let $\ph \in \schur$ have a carapoint at $\tau$. $\ph$ has a regular generalized model if and only if $\tau$ is a differentiable carapoint for $\ph$.
\end{theorem}
\begin{proof}
$(\Rightarrow):$ Suppose that $\ph$ has a regular generalized model $(\M, v, I_Y)$ at $\tau$. From \eqref{decompderiv}, 
\begin{align*}
D_\delta \ph(\tau) = &\ip{\delta^1 E_1 v_\tau}{E_1 v_\tau} + \ip{\delta^2 E_0 v_\tau}{E_0 v_\tau} \notag \\
&+  \ip{\frac{\delta^1\delta^2}{\delta^1 (1-Y_0) + \delta^2 Y_0} E v_\tau}{E v_\tau},
\end{align*}
but as the model is regular, this reduces to 
\[
D_\delta \ph(\tau) = \ip{\delta^1 E_1 v_\tau}{E_1 v_\tau} + \ip{\delta^2 E_0 v_\tau}{E_0 v_\tau}.
\]
Clearly the directional derivative is linear in $\delta$, and thus $\tau$ is a differentiable carapoint for $\ph$.

$(\Leftarrow)$: Assume that $\ph$ has a differentiable carapoint. By Theorem \ref{genmodexists}, there is a generalized model $(\M, v, I_Y)$ of $\ph$. Any expression for the directional derivative will have to be linear in $\delta$, but this means that $P_{\N^\perp}v_\tau = 0$, and so the model is  regular. 
\end{proof} 
	
\bibliography{references}
\bibliographystyle{plain}

\end{document}